\pdfoutput=1 
\RequirePackage[l2tabu, orthodox]{nag} 
\documentclass[a4paper,11pt,reqno]{amsart}

\usepackage[stretch=10]{microtype} 
\usepackage[utf8]{inputenc}
\usepackage[T1]{fontenc}
\usepackage[main=UKenglish,ngerman]{babel}
\usepackage[margin=32mm,bottom=40mm]{geometry}
\usepackage{amsmath,amsthm,amssymb,amsfonts,mathrsfs}
\usepackage{mathtools}
\usepackage{mathdots}
\usepackage{enumitem}
\usepackage{etoolbox} 
\usepackage{verbatim} 

\usepackage{tikz}
\usepackage{tikz-cd} 
\usetikzlibrary{calc}
\usetikzlibrary{matrix,arrows}
\usetikzlibrary{decorations.markings}
\newcommand{\punct}[1]{\makebox[0pt][l]{\,#1}} 

\usepackage{xcolor}

\usepackage{aliascnt} 

\definecolor{dblue}{rgb}{0,0,.6}
\usepackage[pdfauthor={Remy van Dobben de Bruyn, Matthias Paulsen},pdftitle={The construction problem for Hodge numbers modulo an integer in positive characteristic},colorlinks=true,linkcolor=dblue,citecolor=dblue,filecolor=dblue,menucolor=dblue,urlcolor=dblue]{hyperref}

\newtheoremstyle{thms}{1em}{0pt}{\itshape}{}{\bfseries}{.}{ }{} 
\theoremstyle{thms}
\newtheorem{theorem}{Theorem}[section]                      

\newaliascnt{corollary}{theorem}                            
\newtheorem{corollary}[corollary]{Corollary}
\aliascntresetthe{corollary}

\newaliascnt{lemma}{theorem}                                
\newtheorem{lemma}[lemma]{Lemma}
\aliascntresetthe{lemma}

\newaliascnt{proposition}{theorem}                          
\newtheorem{proposition}[proposition]{Proposition}
\aliascntresetthe{proposition}

\newaliascnt{question}{theorem}                             

\aliascntresetthe{question}

\newtheorem{thm}{Theorem}                                   

\newtheoremstyle{defs}{1em}{0pt}{}{}{\bfseries}{.}{ }{} 
\theoremstyle{defs}
\newaliascnt{definition}{theorem}                           
\newtheorem{definition}[definition]{Definition}
\aliascntresetthe{definition}

\newaliascnt{remark}{theorem}                               
\newtheorem{remark}[remark]{Remark}
\aliascntresetthe{remark}

\newaliascnt{setup}{theorem}                                
\newtheorem{setup}[setup]{Setup}
\aliascntresetthe{setup}

\LetLtxMacro\oldproof\proof                                 
\renewcommand{\proof}[1][Proof]{\oldproof[#1]\unskip}

\setlist{itemsep=0em,topsep=0cm,partopsep=0em,parsep=\lineskip}

\setlist[enumerate]{label=\normalfont(\arabic*)}

\newcommand*{\OO}{\ensuremath{\mathscr{O}}}
\newcommand*{\PP}{\ensuremath{\mathbf{P}}}
\newcommand*{\CC}{\ensuremath{\mathbf{C}}}
\newcommand*{\ZZ}{\ensuremath{\mathbf{Z}}}

\newcommand*{\PROJ}{\ensuremath{\mathbf{Proj}}}
\newcommand*{\Sym}{\ensuremath{\mathbf{Sym}}}

\newcommand*{\sTor}{\ensuremath{\mathscr{T}\!or}}

\makeatletter
\patchcmd{\@settitle}{\uppercasenonmath\@title}{}{}{}
\patchcmd{\@settitle}{\bfseries}{\Large\bfseries\scshape}{}{} 
\patchcmd{\@setauthors}{\MakeUppercase}{\large\scshape}{}{} 
\patchcmd{\@setabstracta}{\vskip}{\advance\skip@40\p@ \vskip}{}{} 
\makeatother

\begin{document}

\renewcommand{\sectionautorefname}{Section}

\title[Hodge numbers in positive characteristic]{The construction problem for Hodge numbers modulo an integer in positive characteristic}
\author{Remy van Dobben de Bruyn}
\address{Department of Mathematics \\ Princeton University \\ Fine Hall \\ Washington Road \\ Princeton, NJ 08544 \\ United States of America}
\address{Institute for Advanced Study \\ 1 Einstein Drive \\ Princeton, NJ 08540 \\ United States of America}
\email{rdobben@math.princeton.edu}
\author{Matthias Paulsen}
\address{Institute of Algebraic Geometry \\ Gottfried Wilhelm Leibniz Universit\"at Hannover \\ Welfengarten 1 \\ \makebox{D-30167} Hannover \\ Germany}
\email{paulsen@math.uni-hannover.de}
\date{15 December 2020} 

\begin{abstract}
Let $k$ be an algebraically closed field of positive characteristic.
For any integer $m\ge2$, we show that the Hodge numbers of a smooth projective $k$-variety
can take on any combination of values modulo~$m$, subject only to Serre duality.
In particular, there are no non-trivial polynomial relations between the Hodge numbers.
\end{abstract}

\vspace*{-3em}
\begin{center}
\noindent\makebox[\linewidth]{\rule{16cm}{0.4pt}}
\vspace{8em}

\noindent\makebox[\linewidth]{\rule{13cm}{0.4pt}}
\end{center}

\vspace*{-11.1em}

\maketitle

\vspace*{-1em}

\section{Introduction}
The Hodge numbers $h^{p,q}(X)=\dim_\CC H^q(X,\Omega_X^p)$ of an $n$-dimensional smooth projective
variety $X$ over $\CC$ satisfy the following conditions:
\begin{enumerate}
    \item\label{condition:h00} $h^{0,0}(X)=1$ (connectedness);
    \item\label{condition:serre} $h^{p,q}(X)=h^{n-p,n-q}(X)$ for all $0\le p,q\le n$ (Serre duality);
    \item\label{condition:hodge} $h^{p,q}(X)=h^{q,p}(X)$ for all $0\le p,q\le n$ (Hodge symmetry).
\end{enumerate}
Kotschick and Schreieder showed \cite[Thm.~1, consequence (2)]{kotschick-schreieder} that the only linear relations among
the Hodge numbers that are satisfied by \emph{all} smooth projective $\CC$-varieties of dimension~$n$
are the ones induced by \ref{condition:h00}, \ref{condition:serre}, and \ref{condition:hodge}.

In positive characteristic, Hodge symmetry \ref{condition:hodge} does not always
hold \cite[Prop.~16]{serre}, but Serre duality \ref{condition:serre} is still true.
The first author proved that \ref{condition:h00} and \ref{condition:serre} are indeed
the only universal \emph{linear} relations among the Hodge numbers of $n$-dimensional smooth projective
$k$-varieties if $\operatorname{char} k>0$ \cite[Thm.~1]{vDdB}.

In \cite[Thm.~2]{paulsen-schreieder}, the second author and Schreieder solved the construction
problem over $\CC$ for Hodge diamonds modulo an arbitrary integer $m\ge2$.
This means that for any dimension~$n$ and any collection of integers satisfying the conditions
\ref{condition:h00}, \ref{condition:serre}, and \ref{condition:hodge},
there exists a smooth projective $\CC$-variety of dimension~$n$ whose Hodge numbers
agree with the given integers modulo~$m$. As a corollary, there are no non-trivial \emph{polynomial}
relations among the Hodge numbers, which strengthens the result from \cite{kotschick-schreieder}
on linear relations.

In this paper, we solve the construction problem for Hodge diamonds modulo $m$
in positive characteristic:

\begin{thm}\label{theorem:main}
Let $k$ be an algebraically closed field of positive characteristic,
and let $m\ge2$ and $n\ge0$ be integers.
Let $(a^{p,q})_{0\le p,q\le n}$ be any collection of integers
such that $a^{0,0}=1$ and $a^{p,q}=a^{n-p,n-q}$ for all $0\le p,q\le n$.
Then there exists a smooth projective $k$-variety $X$ of dimension~$n$ such that
\[
h^{p,q}(X)\equiv a^{p,q} \pmod m
\]
for all $0\le p,q\le n$.
\end{thm}

In analogy to \cite[Cor.~3]{paulsen-schreieder}, it follows that there are no polynomial relations
among the Hodge numbers in positive characteristic besides \ref{condition:h00} and \ref{condition:serre}
(see \autoref{corollary:poly}).
This extends the result from \cite[Thm.~1]{vDdB} on linear relations.

\autoref{theorem:main} also shows that Hodge symmetry may fail arbitrarily badly in positive characteristic.
For any dimension~$n$ and all $0\le p<q\le n$ with $p+q\ne n$, not only can the Hodge numbers $h^{p,q}$ and $h^{q,p}$ be different, but they can even be incongruent modulo any integer $m\ge2$.
Note that Hodge symmetry \ref{condition:hodge} is a consequence of Serre duality \ref{condition:serre}
if $p+q=n$, and thus always holds in the middle row of the Hodge diamond.

A complete classification of the possible Hodge diamonds of smooth projective $k$-varieties,
i.\,e.\ a version of \autoref{theorem:main} without the `modulo~$m$' part, seems to be very hard
already when Hodge symmetry is true; see \cite{schreieder} for strong partial results on this
in characteristic zero.

The structure of our proof is similar to \cite{paulsen-schreieder}, with some improvements.
First we solve the construction problem modulo~$m$ for the outer Hodge numbers, i.\,e.\ the Hodge numbers
$h^{p,q}$ with $p\in\{0,n\}$ or $q\in\{0,n\}$ (see \autoref{proposition:outer}).
Then we prove that for any smooth projective $k$-variety, there exists a sequence of blowups
in smooth centres such that the inner Hodge numbers of the blowup,
i.\,e.\ the Hodge numbers $h^{p,q}$ with $1\le p,q\le n-1$,
attain any given values in $\ZZ/m$ satisfying Serre duality \ref{condition:serre}.
Hence we obtain the following result, which might be of independent interest:

\begin{thm}\label{theorem:birational}
Let $k$ be an algebraically closed field of positive characteristic,
and let $m\ge2$ and $n\ge0$ be integers.
Let $X$ be a smooth projective $k$-variety of dimension~$n$ and let $(a^{p,q})_{1\le p,q\le n-1}$
be any collection of integers such that $a^{p,q}=a^{n-p,n-q}$ for all $1\le p,q\le n-1$.
Then there exists a smooth projective $k$-variety $\tilde X$ birational to $X$ such that
\[
h^{p,q}(\tilde X)\equiv a^{p,q} \pmod m
\]
for all $1\le p,q\le n-1$.
\end{thm}

The analogous result in characteristic zero was obtained in \cite[Thm.~5]{paulsen-schreieder}.
The fact that all outer Hodge numbers are birational invariants in positive characteristic
was proven by Chatzistamatiou and R\"ulling \cite[Thm.~1]{chatzistamatiou-rulling},
so \autoref{theorem:birational} is the best possible statement.
Again, it follows that the result from \cite[Thm.~3]{vDdB} on linear birational invariants
extends to polynomials (see \autoref{corollary:poly-birational}).

In analogy with \cite[Thm.~2]{vDdB}, our constructions only need Serre's counterexample
\cite[Prop.~16]{serre} to generate all Hodge asymmetry.
While the structure of our argument is similar to \cite{paulsen-schreieder},
the absence of condition~\ref{condition:hodge} in positive characteristic raises new difficulties
for both the inner and the outer Hodge numbers.
There is a quick proof of \autoref{theorem:birational} assuming embedded resolution of singularities
in positive characteristic, see \autoref{remark:resolution}. The proof we present is similar, but does
a little more work to avoid using embedded resolution. It relies on Maruyama's theory of elementary transformations of vector bundles, which we briefly recall in the \hyperref[section:elementary]{Appendix}.

In \autoref{section:lemmas}, we state and prove some lemmas on Hodge numbers that are used later.
The constructions for outer and inner Hodge numbers are carried out in \autoref{section:outer}
and \autoref{section:inner}, respectively.
Finally, we deduce corollaries on polynomial relations in \autoref{section:poly}.

\subsection*{Notation} Throughout this paper, we fix an algebraically closed field $k$ of positive characteristic
and an integer $m\ge2$.

\section*{Acknowledgements}
{\small
The authors would like to thank Stefan Schreieder for his suggestion to work on this project together.
The first author thanks Raymond Cheng, H\'el\`ene Esnault, and Johan de Jong for useful conversations.
The authors are grateful to J{\k e}drzej Garnek for correcting a mistake in the proof of \autoref{lemma:lefschetz}
and to the referee for pointing out a problem in an earlier version of \autoref{lemma:asymmetric subvarieties}
as well as other helpful suggestions.

The first author is supported by the Oswald Veblen Fund at the Institute for Advanced Study.
The second author is supported by the DFG project `\foreignlanguage{ngerman}{Topologische Eigenschaften von algebraischen Variet\"aten}' (project no.\ 416054549).
The first author thanks the \foreignlanguage{ngerman}{Ludwig-Maximilians-Universit\"at M\"unchen} for their hospitality
during a visit where part of this work was carried out.
}

\numberwithin{equation}{section} 

\section{Some lemmas on Hodge numbers}\label{section:lemmas}
In this section, we collect some standard results on Hodge numbers that we will use repeatedly in the arguments.
The only difference between the situation in characteristic zero \cite{kotschick-schreieder,paulsen-schreieder}
and positive characteristic \cite{vDdB} comes from asymmetry of Hodge diamonds,
and as in \cite{vDdB} the only example we need is Serre's surface:
\begin{theorem}\label{theorem:surface}
There exists a smooth projective $k$-variety $S$ of dimension two such that $h^{1,0}(S)=0$ and $h^{0,1}(S)=1$.
\end{theorem}
\begin{proof}
See \cite[Prop.~16]{serre}, or \cite[Prop.~1.4]{vDdB} for a short modern account.
\end{proof}
\vskip-\lastskip
We use the following well-known formula for Hodge numbers under blowups.
\begin{lemma}\label{lemma:blowup}
Let $X$ be a smooth projective $k$-variety, let $Z \subseteq X$ be a smooth subvariety of codimension~$r$,
and let $\tilde X \to X$ be the blowup of $X$ at $Z$. Then the Hodge numbers of $\tilde X$ satisfy
\[
h^{p,q}(\tilde X) = h^{p,q}(X) + \sum_{i=1}^{r-1} h^{p-i,q-i}(Z).
\]
\end{lemma}

A consequence that will be used repeatedly is that any blowup construction
carried out $m$~times does not change the Hodge numbers modulo~$m$.

\begin{proof}[Proof of \autoref{lemma:blowup}]
See for example \cite[Cor.~IV.1.1.11]{gros}.
As noted by Achinger and Zdanowicz \cite[Cor.~2.8]{achinger-zdanowicz}, it is also an immediate consequence of
Voevodsky's motivic blowup formula \cite[Prop.~3.5.3]{voevodsky} and Chatzistamatiou--R\"ulling's action
of Chow groups on Hodge cohomology \cite{chatzistamatiou-rulling}.
\end{proof}
\vskip-\lastskip
The Hodge numbers of a product $X_1\times X_2$ can be easily described in terms of the Hodge numbers of
$X_1$ and $X_2$ by a K\"unneth-type formula.
\begin{lemma}\label{lemma:kuenneth}
Let $X_1$ and $X_2$ be smooth projective $k$-varieties. Then the Hodge numbers of $X:=X_1\times X_2$
are given by
\[
h^{p,q}(X) = \sum_{\substack{p_1+p_2=p\\q_1+q_2=q}}h^{p_1,q_1}(X_1)\cdot h^{p_2,q_2}(X_2).
\]
\end{lemma}
\begin{proof}
We have $\Omega_X=\pi_1^*\Omega_{X_1}\oplus\pi_2^*\Omega_{X_2}$ and thus
\[ \Omega^p_X=\bigoplus_{p_1+p_2=p}\pi_1^*\Omega^{p_1}_{X_1}\otimes\pi_2^*\Omega^{p_2}_{X_2}. \]
Hence, using the classical K\"unneth formula for quasi-coherent sheaves we get
\begin{align*}
H^q(X,\Omega^p_X)
&= \bigoplus_{p_1+p_2=p}H^q\left(X,\pi_1^*\Omega^{p_1}_{X_1}\otimes\pi_2^*\Omega^{p_2}_{X_2}\right) \\
&= \bigoplus_{\substack{p_1+p_2=p\\q_1+q_2=q}}H^{q_1}\left(X_1,\Omega^{p_1}_{X_1}\right)\otimes H^{q_2}\left(X_2,\Omega^{p_2}_{X_2}\right).
\qedhere
\end{align*}
\end{proof}
\vskip-\lastskip
The next lemma provides a weak Lefschetz theorem for sufficiently ample hypersurfaces.
\begin{lemma}\label{lemma:lefschetz}
Let $X$ be a smooth projective $k$-variety of dimension~$n+1$ with a very ample line bundle
$\mathscr L = \OO_X(H)$. Let $d_0 \in \ZZ_{>0}$ such that $H^q(X,\Omega^p_X(-dH)) = 0$ when
$d \geq d_0$ and $p + q \leq n$. Then any smooth divisor $Y \in |dH|$ with $d \geq d_0$ satisfies
$h^{p,q}(Y)=h^{p,q}(X)$ when $p+q\le n-1$.
\end{lemma}
\begin{proof}
The short exact sequence
\[
0 \to \Omega^p_X(-dH) \to \Omega^p_X \to \Omega^p_X\big|_Y \to 0
\]
shows that for all $p + q \leq n - 1$ and all $e \geq 0$, we have
\begin{equation}
H^q\Big(X, \Omega^p_X(-eH)\Big) = H^q\Big(Y, \Omega^p_X(-eH)\big|_Y\Big).\label{eq:restriction}
\end{equation}
We will prove by induction on $p$ that $H^q(Y, \Omega^p_X(-eH)|_Y) = H^q(Y, \Omega^p_Y(-eH))$
for all $e \geq 0$ and $p + q \leq n - 1$.
Together with \eqref{eq:restriction} this proves the result by taking $e = 0$.
The base case $p = 0$ is trivial since $\OO_X|_Y = \OO_Y$. For $p > 0$, the inductive hypothesis,
\eqref{eq:restriction}, and the assumption on $d_0$ imply
\begin{equation}
H^q\Big(Y, \Omega^i_Y(-eH)\Big) = H^q\Big(Y, \Omega^i_X(-eH)\big|_Y\Big)
= H^q\Big(X, \Omega^i_X(-eH)\Big) = 0 \label{eq:inherited vanishing}
\end{equation}
for $i + q \leq n - 1$, $e \geq d_0$, and $i < p$. The conormal sequence
\[
0 \to \OO_Y(-Y) \to \Omega^1_X\big|_Y \to \Omega^1_Y \to 0
\]
gives a short exact sequence
\begin{equation}
0 \to \Omega^{p-1}_Y(-Y) \to \Omega^p_X\big|_Y \to \Omega^p_Y \to 0\label{eq:exterior powers}
\end{equation}
since $\OO_Y(-Y)$ is a line bundle.
Now \eqref{eq:inherited vanishing} gives
\begin{equation*}
H^q\Big(Y,\Omega^{p-1}_Y(-Y-eH)\Big) = H^q\Big(Y,\Omega^{p-1}_Y\left(-(d+e)H\vphantom{\big(}\right)\Big) = 0
\end{equation*}
for $p + q \leq n$ and $e \geq 0$.
Thus, \eqref{eq:exterior powers} shows that the natural map
\[
H^q\Big(Y,\Omega^p_X(-eH)\big|_Y\Big) \to H^q\Big(Y,\Omega^p_Y(-eH)\Big)
\]
is an isomorphism for $p + q \leq n - 1$ and $e \geq 0$, as claimed.
\end{proof}
\begin{corollary}\label{corollary:lefschetz}
Let $X$ be a smooth projective $k$-variety of dimension~$n+1$ with a very ample line bundle $\mathscr L = \OO_X(H)$.
Then any smooth divisor $Y \in |dH|$ with $d \gg 0$ satisfies $h^{p,q}(Y)=h^{p,q}(X)$ when $p+q\le n-1$.
\end{corollary}
\begin{proof}
By Serre vanishing, there exists $d_0 \in \ZZ$ such that $H^q(X,\Omega^p_X(-dH)) = 0$ for all $d \geq d_0$
and $q \leq n$. Then \autoref{lemma:lefschetz} gives the result.
\end{proof}
\begin{remark}
If $\operatorname{char} k = 0$, then by Nakano vanishing we may take $d_0 = 1$ in \autoref{lemma:lefschetz}.
This recovers the usual proof of weak Lefschetz from Nakano vanishing, although usually the implication
goes in the other direction. Similarly, if $\operatorname{char} k > 0$ and Nakano vanishing holds for $X$,
then we may take $d_0 = 1$, but in general Kodaira vanishing may already fail in positive characteristic \cite{raynaud}.
\end{remark}
For our application, it's useful to have some control over the Euler characteristic of~$\mathscr L^{-1}$.
\begin{lemma}\label{lemma:chi}
Let $X$ be a smooth projective $k$-variety of dimension $n+1$ and let $e\in\ZZ$.
Then, up to modifying $X$ by blowups in smooth centres that do not change its Hodge numbers modulo~$m$,
we may assume that $X$ admits a very ample line bundle $\mathscr L = \OO_X(H)$ such that
$\chi(X,\mathscr L^{-1})\equiv e\pmod m$ and such that any smooth divisor $Y \in |H|$ satisfies
$h^{p,q}(Y) = h^{p,q}(X)$ when $p + q \leq n - 1$.
\end{lemma}
\begin{proof}
Let $\pi \colon \tilde X \to X$ be a blowup in $m$ distinct points $p_1,\ldots,p_m \in X$.
Then the blowup formula for Hodge numbers (\autoref{lemma:blowup}) gives
$h^{p,q}(\tilde X) \equiv h^{p,q}(X) \pmod m$.
Let $E_i = \pi^{-1}(p_i)$ be the exceptional divisors, and for $r \in \{0,\ldots,m\}$ write
$E_{\leq r} = E_1 + \ldots + E_r$. Then the short exact sequence
\[
0 \to \OO_{\tilde X}(- E_{\leq r}) \to \OO_{\tilde X} \to \OO_{E_{\leq r}} \to 0
\]
shows that
\[
\chi\big(\tilde X, \OO_{\tilde X}(-E_{\leq r})\big) = \chi(\tilde X, \OO_{\tilde X}) - \sum_{i=1}^r \chi(E_i, \OO_{E_i}) = \chi(X, \OO_X) - r.
\]
Take $r \in \{0,\ldots,m-1\}$ with $r \equiv \chi(X,\OO_X) - e \pmod m$.

Let $\mathscr M$ be an ample line bundle on $\tilde X$.
By Serre vanishing there exists $a_0 \in \ZZ$ such that for all $a \geq a_0$, the line bundle
$\mathscr L = \mathscr M^{\otimes a} \otimes \OO_{\tilde X}(E_{\leq r})$ is very ample and satisfies
\begin{equation}
H^q\big(X,\Omega^p_X\otimes \mathscr L^{-d}\big) = 0\label{eq:nakano}
\end{equation}
for $d > 0$ and $q \leq n$. Taking $a$ divisible by the product of $m$ and the denominators of the coefficients
of the Hilbert polynomial $P(t) = \chi(\tilde X,\mathscr M^{\otimes t} \otimes \OO_{\tilde X}(-E_{\leq r}))$,
we see that
\[
\chi\big(\tilde X, \mathscr L^{-1}\big) \equiv \chi\big(\tilde X, \OO_{\tilde X}(-E_{\leq r})\big) \equiv e \pmod m.
\]
Finally, $\mathscr L$ satisfies weak Lefschetz by \eqref{eq:nakano} and \autoref{lemma:lefschetz}.
\end{proof}

\section{Outer Hodge numbers}\label{section:outer}

In this section, we solve the construction problem for the outer Hodge numbers. Because of Serre duality
and the fact that $h^{0,0}=1$, it suffices to consider the Hodge numbers $h^{p,q}$ with $(p,q)\in J_n$,
where
\[
J_n = \{(1,0),\ldots,(n,0),(0,1),\ldots,(0,n)\} .
\]
The main result of this section is the following:
\begin{proposition}\label{proposition:outer}
Let $n\ge0$.
For any given integers $a^{1,0},\ldots,a^{n,0}$ and $a^{0,1},\ldots,a^{0,n}$ with $a^{n,0}=a^{0,n}$,
there exists a smooth projective $k$-variety $X$ of dimension~$n$ such that
\[
h^{p,q}(X)\equiv a^{p,q}\pmod m
\]
for all $(p,q) \in J_n$.
\end{proposition}
The construction will be carried out by induction on the dimension,
using the weak Lefschetz results from \autoref{corollary:lefschetz} and \autoref{lemma:chi}.
\begin{lemma}\label{lemma:outer-mid}
Let $n,d\geq 0$ be integers such that $d \geq n-1$.
If \autoref{proposition:outer} holds in dimension~$d$ for
$a^{1,0},\ldots,a^{d,0}$ and $a^{0,1},\ldots,a^{0,d}$ with $a^{d,0}=a^{0,d}$,
then it also holds in dimension~$n$ for $a^{1,0},\ldots,a^{n-1,0},b$ and $a^{0,1},\ldots,a^{0,n-1},b$
for any $b \in \ZZ$.
\end{lemma}
\begin{proof}
Let $X$ be a smooth projective $k$-variety of dimension~$d$ with the given Hodge numbers $a^{p,q}$.
We may assume that $d\ge n+1$ by multiplying $X$ with $\PP^2$,
which does not change its outer Hodge numbers in degree $\le n-1$.
By repeatedly replacing $X$ by a smooth hyperplane section of sufficiently high degree,
we may further assume that $d = n + 1$ by \autoref{corollary:lefschetz}.
By \autoref{lemma:chi}, after possibly replacing $X$ by a blowup that does not change
its Hodge numbers modulo $m$, there exists a very ample line bundle $\mathscr L$ on $X$ such that
\begin{equation}
\chi(X,\mathscr L^{-1}) \equiv (-1)^n(a^{0,n}-a^{0,n+1}-b) \pmod m\label{eq:euler characteristic}
\end{equation}
and such that a smooth section $Y$ of $\mathscr L$ satisfies $h^{p,q}(Y)\equiv a^{p,q}\pmod m$ for $p+q\le n-1$.
The short exact sequence
\[
0 \to \mathscr L^{-1} \to \OO_X \to \OO_Y \to 0
\]
gives $\chi(X,\mathscr L^{-1}) = \chi(X,\OO_X) - \chi(Y,\OO_Y)$.
Since $h^{0,q}(X) = h^{0,q}(Y)$ for $q \leq n-1$, we conclude that
\begin{align*}
\chi(X,\mathscr L^{-1}) &= (-1)^n h^{0,n}(X) + (-1)^{n+1} h^{0,n+1}(X) - (-1)^n h^{0,n}(Y)\\
&\equiv (-1)^n \left(a^{0,n} - a^{0,n+1} - h^{0,n}(Y)\right). \pmod m
\end{align*}
With \eqref{eq:euler characteristic} we get $h^{0,n}(Y) \equiv b \pmod m$,
so Serre duality gives $h^{n,0}(Y) \equiv b \pmod m$.
\end{proof}
\vskip-\lastskip
Note that in characteristic zero, \autoref{lemma:outer-mid} immediately implies \autoref{proposition:outer},
giving an alternative approach to a variant of \cite[Prop.~4]{paulsen-schreieder}.
In positive characteristic, however, the failure of Hodge symmetry raises new difficulties,
since e.\,g.\ $h^{n-1,0}=h^{0,n-1}$ is true for varieties of dimension~$n-1$
but not for all varieties of dimension~$n$.
This problem is solved in the following construction,
which together with \autoref{lemma:outer-mid} implies \autoref{proposition:outer}.
\begin{lemma}\label{lemma:outer-low}
Let $n\ge2$. For any given integers $a^{0,1},\ldots,a^{0,n-1}$ and $a^{1,0},\ldots,a^{n-1,0}$,
there exists a smooth projective $k$-variety $X$ of dimension $\ge n-1$ such that
\[
h^{p,q}(X)\equiv a^{p,q}\pmod m
\]
for all $(p,q)\in J_{n-1}$.
\end{lemma}
Note that we do not assume $a^{0,n-1}=a^{n-1,0}$ here, so we typically need $\dim X \geq n$.
\begin{proof}[Proof of \autoref{lemma:outer-low}]
First consider the case $n=2$.
Let $E$ be an elliptic curve and let~$S$ be the surface from \autoref{theorem:surface}.
Choose $i\ge0$ and $j\ge1$ with $i\equiv a^{0,1}-a^{1,0}\pmod m$ and $j\equiv a^{1,0}\pmod m$, and
set $X=S^i\times E^j$. Then it follows from K\"unneth's formula (\autoref{lemma:kuenneth}) that
$h^{0,1}(X)\equiv i+j\equiv a^{0,1}\pmod m$ and $h^{1,0}(X)\equiv j\equiv a^{1,0}\pmod m$.

Now assume $n \geq 3$.
By \autoref{lemma:outer-mid}, we may assume inductively that
\autoref{proposition:outer} holds in dimensions $\leq n-1$.
Therefore, there exists a smooth projective variety $Y$ of dimension $n-1$ with outer Hodge numbers
\[
h^{p,q}(Y) \equiv \begin{cases}
(-1)^q, & p = 0,\ 0 \leq q < n-1,\\
0, & p = 0,\ q = n - 1,\\
0, & p > 0,\ q = 0.
\end{cases} \pmod m.
\]
By \autoref{proposition:outer} in dimension~$2$, there exists a smooth projective surface $S$ with outer Hodge numbers
$h^{1,0}(S)\equiv h^{2,0}(S)\equiv h^{0,2}(S)\equiv 0\pmod m$ and $h^{0,1}(S)\equiv 1\pmod m$.
The K\"unneth formula from \autoref{lemma:kuenneth} shows that $S \times Y$ has outer Hodge
numbers $h^{p,q}(S \times Y) \equiv 0\pmod m$ for $(p,q) \in J_{n-1}$,
except $h^{0,0}(S \times Y) = 1$ and $h^{0,n-1}(S \times Y) \equiv (-1)^n \pmod m$.

Finally, by \autoref{proposition:outer} in dimension $n-1$, there exists a smooth projective variety $Z$
with outer Hodge numbers given by
\[
h^{p,q}(Z) \equiv \begin{cases}
a^{p,q}, & (p,q) \in J_{n-1} \setminus\{(0,n-1)\}, \\
a^{n-1,0}, & (p,q) = (0,n-1).
\end{cases} \pmod m
\]
Taking $X = Z\times(S\times Y)^i$ for $i \geq 0$ gives outer Hodge numbers
\[
h^{p,q}(X) \equiv \begin{cases}
a^{p,q}, & (p,q) \in J_{n-1} \setminus\{(0,n-1)\}, \\
a^{n-1,0} + (-1)^n i, & (p,q) = (0,n-1).
\end{cases} \pmod m
\]
The result follows by taking $i \equiv (-1)^n(a^{0,n-1}-a^{n-1,0}) \pmod m$.
\end{proof}

\section{Inner Hodge numbers}\label{section:inner}

The aim of this section is to prove \autoref{theorem:birational}, i.\,e.\ to modify the inner Hodge numbers
of a smooth projective $k$-variety via successive blowups.
We first show how to produce certain subvarieties with asymmetric Hodge numbers that we will blow up later.
\begin{lemma}\label{lemma:asymmetric subvarieties}
Let $X$ be a smooth projective $k$-variety of dimension~$n$,
let $b,c\in\ZZ$, and let $d \in \{2,\ldots,n-2\}$. Then there exists a smooth projective variety $\tilde X$
and a birational morphism $\tilde X \to X$ obtained as a composition of blowups in smooth centres
that does not change the Hodge numbers modulo~$m$ such that
$\tilde X$ contains a smooth subvariety $W$ of dimension~$d$ satisfying
\begin{equation}
h^{d,0}(W) = h^{0,d}(W) \equiv 0 \pmod m \label{eq:W-zero}
\end{equation}
and
\begin{equation}
h^{d-1,0}(W)\equiv b, \quad h^{0,d-1}(W)\equiv c \pmod m. \label{eq:W-bc}
\end{equation}
\end{lemma}

\begin{proof}
Let $X_1 \to X$ be the blowup of $X$ in a point.
The assumption on $d$ implies $n \geq 4$,
so the exceptional divisor of $X_1$ contains $\PP^3$.
By \autoref{proposition:outer}, there exists a smooth projective surface $S_0$ such that
$h^{2,0}(S_0)=h^{0,2}(S_0)\equiv0\pmod m$ and
\[
h^{1,0}(S_0)\equiv b, \quad h^{0,1}(S_0)\equiv c \pmod m.
\]
Choose a possibly singular surface $S_1 \subseteq \PP^3$ birational to $S_0$.
By embedded resolution of surfaces \cite[Thm.~9.1.3]{abhyankar} (see also \cite[Thm.~1.2]{cutkosky}),
there exists a birational morphism $X_2 \to X_1$ obtained as a composition of blowups in smooth centres
contained in $\PP^3$ such that the strict transform $S$ of $S_1$ is smooth.
Since $S$ is also birational to $S_0$, we have
$h^{2,0}(S)=h^{0,2}(S)\equiv0\pmod m$ and
\[
h^{1,0}(S)\equiv b, \quad h^{0,1}(S)\equiv c \pmod m.
\]
Now consider the blowup $X_3 \to X_2$ in $S$. The exceptional divisor is a $\PP^{n-3}$-bundle
$\PP_S(\mathscr E)$ over $S$. Let $Z \subseteq \PP^{n-3}$ be a smooth hypersurface of degree $d$ in a linear
subspace $\PP^{d-1} \subseteq \PP^{n-3}$; in particular, $Z$ satisfies $h^{d-2,0}(Z) = h^{0,d-2}(Z) = 1$.

By Maruyama's theory of elementary transformations (see \autoref{prop:elementary}), there exists a diagram
\nopagebreak\vspace{-1em}
\begin{equation*}
\begin{tikzcd}[column sep=.1em,row sep=2.3em]
& \widetilde{P}
\ar[start anchor={[xshift=.5em]south west},end anchor={[xshift=-1.4em]north east}]{ld}[swap]{f}
\ar[start anchor={[xshift=-.5em]south east}, end anchor={[xshift=1.4em]north west}]{rd}{f'} & & & & \\
S \times \PP^{n-3} & & \PP_S(\mathscr E) \ar[hook]{rrr} & \makebox{} & \makebox{} & X_3\punct{,}
\end{tikzcd}
\end{equation*}
where $f$ and $f'$ are blowups in smooth centres $Y$ and $Y'$ respectively, such that $Y \cap(S \times Z)$ is smooth.
Then the blowup $X_4 \to X_3$ in $Y'$ contains the strict transform
\[
W = \widetilde{S \times Z} = \operatorname{Bl}_{Y \cap (S \times Z)}(S \times Z)
\]
of $S \times Z$ under $f$.
Birational invariance of outer Hodge numbers (in the case of a blowup this is \autoref{lemma:blowup})
and the K\"unneth formula (\autoref{lemma:kuenneth}) give
\begin{align*}
h^{d,0}(W) &= h^{0,d}(W) = h^{d,0}(S \times Z) = h^{2,0}(S)h^{d-2,0}(Z) \equiv 0 \pmod m, \\
h^{d-1,0}(W) &= h^{d-1,0}(S \times Z) = h^{2,0}(S)h^{d-3,0}(Z) + h^{1,0}(S)h^{d-2,0}(Z) \equiv b \pmod m, \\
h^{0,d-1}(W) &= h^{0,d-1}(S \times Z) = h^{0,2}(S)h^{0,d-3}(Z) + h^{0,1}(S)h^{0,d-2}(Z) \equiv c \pmod m.
\end{align*}
Blowing up $m-1$ more points coming from $X$ and repeating the above construction $m-1$ more times
in each exceptional $\PP^{n-1}$ separately, the blowup formula of \autoref{lemma:blowup} shows
that the Hodge numbers of $X$ do not change modulo~$m$.
\end{proof}

\begin{corollary}\label{corollary:asymmetric-hodge}
Let $X$ be a smooth projective $k$-variety of dimension~$n$, let $b, c \in \ZZ$, and let $r \in \{1,\ldots,n-1\}$.
Assume that $b=c$ if $r=1$ or $r=n-1$.
Then there exists a birational morphism $\tilde X \to X$
obtained by a sequence of blowups in smooth centres such that
\[
h^{r,1}(\tilde X) \equiv h^{r,1}(X)+b, \quad h^{1,r}(\tilde X)\equiv h^{1,r}(X)+c \pmod m
\]
and
\[
h^{p,1}(\tilde X) \equiv h^{p,1}(X), \quad h^{1,p}(\tilde X)\equiv h^{1,p}(X)\pmod m
\]
for all $p > r$.
\end{corollary}

\begin{proof}
If $r\in\{2,\ldots,n-2\}$, then \autoref{lemma:asymmetric subvarieties} shows that
there exists a successive blowup $X' \to X$ that does not change the Hodge numbers modulo~$m$
such that $X'$ contains a subvariety $W$ of dimension~$r$ satisfying \eqref{eq:W-zero} and \eqref{eq:W-bc}.
Letting $\tilde X \to X'$ be the blowup in $W$ gives the result by \autoref{lemma:blowup}.

For $r=1$, we consider the blowup in $i\ge0$ points where $i\equiv b=c\pmod m$. Then the statement follows
again from \autoref{lemma:blowup}.

For $r=n-1$, we first blow up $X$ in $i\ge0$ points where $i\equiv b=c\pmod m$.
Then, in each exceptional $\PP^{n-1}$ we blow up a smooth hypersurface $Z$ of degree~$n$.
Since $h^{n-2,0}(Z)=h^{0,n-2}(Z)=1$, the result follows from \autoref{lemma:blowup}.
\end{proof}
\vskip-\lastskip
We are now able to solve the construction problem modulo~$m$ for the second outer Hodge numbers,
i.\,e.\ the inner Hodge numbers $h^{p,q}$ with $p\in\{1,n-1\}$ or $q\in\{1,n-1\}$,
via repeated blowups in smooth centres.
By Serre duality, it is enough to consider the Hodge numbers $h^{p,q}$ with
$(p,q)\in I_n$, where
\[
I_n=\left\{(1,q)\ \Big|\ q\in\{1,\ldots,n-1\}\right\}\cup\left\{(p,1)\ \Big|\  p\in\{1,\ldots,n-1\}\right\}.
\]
\begin{corollary}\label{corollary:second-outer}
Let $X$ be a smooth projective $k$-variety of dimension~$n$.
For any given collection of integers $(a^{p,q})_{(p,q)\in I_n}$ with $a^{n-1,1}=a^{1,n-1}$,
there exists a birational morphism $\tilde X\to X$ obtained by a sequence of blowups in smooth centres such that
\[
h^{p,q}(\tilde X)\equiv a^{p,q} \pmod m
\]
for all $(p,q)\in I_n$.
\end{corollary}
\begin{proof}
For $r\in\{1,\ldots,n-1\}$, let $b=a^{r,1}-h^{r,1}(X)$ and $c=a^{1,r}-h^{1,r}(X)$.
We see that $b=c$ if $r=1$ or $r=n-1$.
Hence, we may apply \autoref{corollary:asymmetric-hodge} for all $r\in\{1,\ldots,n-1\}$
in descending order to obtain the result.
\end{proof}
\vskip-\lastskip
Finally, we are ready to prove \autoref{theorem:birational}, which together with \autoref{proposition:outer}
implies our main result \autoref{theorem:main}.
\begin{proof}[Proof of \autoref{theorem:birational}]
We will proceed by induction on $n$.
The case $n \leq 1$ is vacuous, as there are no inner Hodge numbers.
Let $n \geq 2$, and assume the result is known in all dimensions $\leq n-1$.
By \autoref{corollary:second-outer}, there exists a birational morphism $X_1\to X$
obtained by a sequence of blowups in smooth centres such that for $(p,q)\in I_n$ we have
\[
h^{p,q}(X_1) \equiv a^{p,q} - h^{p-1,q-1}\left(\PP^{n-2}\right)
\pmod m.
\]
Let $X_2 \to X_1$ be the blowup in a point,
and let $\PP^{n-2} \subseteq X_2$ be a hyperplane in the exceptional divisor.
By the induction hypothesis, there exists a birational morphism $\tilde P \to \PP^{n-2}$
obtained by a sequence of blowups in smooth centres such that the Hodge numbers of $\tilde P$ are given by
\[
h^{p,q}(\tilde P) \equiv \begin{cases}
h^{p,q}\left(\PP^{n-2}\right), & p\in\{0,n-2\}\text{ or }q\in\{0,n-2\}, \\
a^{p+1,q+1} - h^{p+1,q+1}(X_1), & \text{else.}
\end{cases} \pmod m
\]
Since $\tilde P \to \PP^{n-2}$ is a sequence of blowups in smooth centres, we can blow up the
(strict transforms of) the same centres in $X_2$ to get a birational morphism $X_3 \to X_2$
such that the strict transform of $\PP^{n-2}$ is $\tilde P$.
Blowing up $m-1$ more points coming from $X_1$
and applying the same construction in each of the exceptional divisors separately
gives a birational morphism $X_4 \to X_1$ that does not change the Hodge numbers modulo~$m$
by the blowup formula of \autoref{lemma:blowup}.
Finally, if we let $\tilde X \to X_4$ be the blowup in one of the $\tilde P$ obtained in this way,
we get
\[
h^{p,q}(\tilde X) = h^{p,q}(X_1) + h^{p-1,q-1}(\tilde P) \equiv a^{p,q} \pmod m
\]
for all $(p,q)$ with $1 \leq p, q \leq n - 1$, which finishes the induction step.
\end{proof}

\begin{remark}\label{remark:resolution}
The proof above can be simplified if one assumes embedded resolution of singularities in arbitrary dimension.
Indeed, by blowing up a finite number of points, we may assume that $h^{1,1}(X) \equiv a^{1,1} - 1 \pmod m$
and $X$ contains $\PP^{n-1}$. Now we claim that we can construct an $(n-2)$-dimensional subvariety $Y$
in a blowup $X' \to X$ with $h^{p,q}(X') \equiv h^{p,q}(X) \pmod m$ such that $h^{p,q}(Y) \equiv a^{p+1,q+1} - h^{p+1,q+1}(X) \pmod m$.
Then the blowup $\tilde X \to X'$ in $Y$ has the required Hodge numbers.

To construct $Y$, first construct any smooth projective variety $Z$ of dimension $n - 2$ 
with the correct outer Hodge numbers using \autoref{proposition:outer}.
Then $Z$ is birational to a (possibly singular) hypersurface $Z' \subseteq \PP^{n-1}$. Embedded resolution
of $Z' \subseteq \PP^{n-1}$ gives a birational map $X' \to X$ such that the strict transform of $Z'$ is smooth,
so $Z'$ has the desired outer Hodge numbers by \cite[Thm.~1]{chatzistamatiou-rulling}.
By the induction hypothesis we may blow up further to get the inner Hodge numbers we want.
Repeating this construction $m-1$ more times, as usual, gives $h^{p,q}(X') \equiv h^{p,q}(X) \pmod m$.

However, because resolution of singularities is currently unknown in positive characteristic beyond dimension~$3$,
we have developed the above approach using embedded resolution of surfaces, Maruyama's theory of elementary
transformations of projective bundles, and the fortuitous fact that the failure of Hodge symmetry is 
`generated' by surfaces (see also \cite[Thm.~2]{vDdB}).
\end{remark}

\begin{remark}
Both the proof of \autoref{theorem:birational} above (replacing \autoref{lemma:asymmetric subvarieties}
by an easy case of \cite[Lem.~6]{paulsen-schreieder}) and the alternative argument of
\autoref{remark:resolution} using resolution of singularities give new methods to prove the
characteristic zero result \cite[Thm.~5]{paulsen-schreieder}.

Conversely, it is possible to adapt the methods of \cite[\S 3]{paulsen-schreieder} to prove
\autoref{theorem:birational}, using the subvarieties from \cite[Lem.~6]{paulsen-schreieder} as well as
projective bundles over the subvarieties from \autoref{lemma:asymmetric subvarieties},
but the analysis is a bit more intricate.
\end{remark}

\section{Polynomial relations}\label{section:poly}

\begin{corollary}\label{corollary:poly}
There are no polynomial relations among the Hodge numbers of smooth projective $k$-varieties of the
same dimension besides the ones induced by Serre duality.
\end{corollary}
\begin{proof}
Using \cite[Lem.~8]{paulsen-schreieder}, this follows from \autoref{theorem:main} in the same way as
\cite[Cor.~3]{paulsen-schreieder},
except that we now consider the Hodge numbers $h^{p,q}$ with $0\le p\le q\le n$ and $(p,q)\ne(0,0),(n,n)$.
\end{proof}

\begin{corollary}\label{corollary:poly-birational}
There are no polynomial relations among the inner Hodge numbers of smooth projective $k$-varieties
of any fixed birational equivalence class besides the ones induced by Serre duality.
\end{corollary}
\begin{proof}
This follows from \autoref{theorem:birational} in a similar fashion.
\end{proof}

\appendix
\phantomsection
\section*{Appendix. Elementary transformations of vector bundles}\label{section:elementary}
\refstepcounter{section}
We include a quick coordinate-free proof of Maruyama's theory of elementary transformations
of vector bundles \cite{maruyama,maruyama2}. See \autoref{thm:maruyama} for the main result, and \autoref{prop:elementary} for the example that we will use.

\begin{setup}\label{setup Maruyama}
Let $S$ be a scheme and $D \subseteq S$ a Cartier divisor. We will consider a vector bundle $\mathscr E$
on $S$ together with a quotient bundle $\mathscr E|_D \twoheadrightarrow \mathscr F$. Write $\mathscr E'$
and $\mathscr F'$ for the kernels of $\mathscr E \to \mathscr F$ and $\mathscr E|_D \to \mathscr F$
respectively, so we get a commutative diagram
\begin{equation}\label{dia Maruyama}
\begin{tikzcd}
0 \ar{r} & \mathscr E' \ar{r}\ar{d} & \mathscr E \ar{r}\ar{d} & \mathscr F \ar{r}\ar[equal]{d} & 0\\
0 \ar{r} & \mathscr F' \ar{r} & \mathscr E\big|_D \ar{r} & \mathscr F \ar{r} & 0\punct{.}
\end{tikzcd}
\end{equation}
Write $\pi \colon X = \PP_S(\mathscr E) \to S$, with tautological quotient line bundle
$\pi^* \mathscr E \twoheadrightarrow \OO_\pi(1)$. The surjection $\mathscr E \twoheadrightarrow \mathscr F$
induces a closed immersion $i \colon Y = \PP_D(\mathscr F) \hookrightarrow \PP_S(\mathscr E)$.
Let $f \colon \tilde X \to X$ be the blowup of $X$ in $Y$ with exceptional divisor $E = f^{-1}(Y)$,
and set $\tilde \pi = \pi \circ f$. The preimage $\tilde \pi^{-1}(D)$ consists of $E$ and $\widetilde{X_D}$,
whose intersection is the exceptional divisor of $\widetilde{X_D} \to X_D = \PP_D(\mathscr E|_D)$.
\end{setup}

\begin{lemma}\label{lemma:centre}
Let $X$, $D$, and $\mathscr E \twoheadrightarrow \mathscr F$ be as in \autoref{setup Maruyama}.
Then $Y \subseteq X$ is cut out by the image of the composite map
\[
(\pi^*\mathscr E') \otimes \OO_\pi(-1) \to (\pi^*\mathscr E) \otimes \OO_\pi(-1) \twoheadrightarrow \OO_X.
\]
\end{lemma}

\begin{proof}
The map $Y \hookrightarrow X$ is given by applying $\PROJ_S$ to the surjection of $\OO_S$-algebras
\[
\Sym^*_{\OO_S} \mathscr E \twoheadrightarrow \Sym^*_{\OO_D} \mathscr F.
\]
The quotient of $\Sym^*_{\OO_S} \mathscr E$ by the ideal generated by $\mathscr E'(-1) \subseteq \Sym^*_{\OO_S} \mathscr E$
is $\Sym^*_{\OO_S} \mathscr F$, which coincides with $\Sym^*_{\OO_D} \mathscr F$ in all
positive degrees since $\mathscr F$ is supported on $D$. The result follows since a morphism of
graded algebras that is eventually an isomorphism induces an isomorphism on $\PROJ$.
\end{proof}

\begin{corollary}\label{cor:universal property}
Let $X$, $D$, and $\mathscr E \twoheadrightarrow \mathscr F$ be as in \autoref{setup Maruyama}, and let
$q \colon T \to S$ be a morphism of schemes. Then morphisms $T \to \tilde X$ of $S$-schemes correspond
to pairs $(\mathscr L, \phi)$ of a line bundle $\mathscr L$ on $T$ and a surjection
$\phi \colon q^* \mathscr E \twoheadrightarrow \mathscr L$, up to isomorphism under $q^* \mathscr E$,
such that the image of the composite map
\[
q^* \mathscr E' \to q^* \mathscr E \twoheadrightarrow \mathscr L
\]
is an invertible subsheaf of $\mathscr L$.
\end{corollary}

\begin{proof}
This follows from the universal properties of projective bundles and blowups.
\end{proof}
\vskip-\lastskip
The basic duality of the situation is captured by the following lemma:

\begin{lemma}\label{lemma:duality}
Let $S$, $D$, and $\mathscr E \twoheadrightarrow \mathscr F$ be as in \autoref{setup Maruyama}.
Then $\mathscr E'$ (resp.\ $\mathscr F'$) is locally free on $S$ (resp.\ $D$),
and $\mathscr E' \twoheadrightarrow \mathscr F'$ is another instance of \autoref{setup Maruyama}.
Applying this operation twice gives $\mathscr E(-D) \twoheadrightarrow \mathscr F(-D)$.
\end{lemma}

\begin{proof}
Since $\mathscr F$ is a quotient bundle of $\mathscr E|_D$, it is clear that $\mathscr F'$ is locally free on $D$.
Moreover, since $\mathscr F$ has $\sTor$ dimension $1$ on $S$, we see that $\mathscr E'$ is a vector bundle.
Applying the snake lemma to \eqref{dia Maruyama} shows that the kernel of $\mathscr E' \to \mathscr F'$
is $\mathscr E(-D)$. Applying $\sTor_*^{\OO_S}(-, \OO_D)$ to the first row
of \eqref{dia Maruyama} gives the exact sequence
\[
0 \to \mathscr F(-D) \to \mathscr E'\big|_D \to \mathscr E\big|_D \to \mathscr F \to 0,
\]
which shows that the kernel of $\mathscr E'|_D \to \mathscr F'$ is $\mathscr F(-D)$. We omit the verification that the map $\mathscr E(-D) \twoheadrightarrow \mathscr F(-D)$ is obtained from the original one by twisting with $\OO_S(-D)$.
\end{proof}
\vskip-\lastskip
In analogy with the notation of \autoref{setup Maruyama}, write $\pi' \colon X' = \PP_S(\mathscr E') \to S$,
with closed subscheme $Y' = \PP_D(\mathscr F')$, and blowup $f' \colon \tilde X' \to X'$ in $Y'$
with exceptional divisor $E' = f'^{-1}(Y')$.

Finally, write $X(-D) = \PP_S(\mathscr E(-D))$, with closed subscheme $Y(-D) = \PP_D(\mathscr F(-D))$ 
and blowup $\tilde X(-D) \to X(-D)$ in $Y(-D)$. The natural isomorphisms $X(-D) \cong X$ and
$Y(-D) \cong Y$ lift to a natural isomorphism $\tilde X(-D) \cong \tilde X$, described in terms of
\autoref{cor:universal property} by $(\mathscr L,\phi) \mapsto (\mathscr L(D),\phi(D))$.

The duality of \autoref{lemma:duality} directly implies the main theorem of elementary transformations of vector bundles \cite[Thm.~1.1, Thm.~1.3]{maruyama}:

\begin{theorem}[Maruyama]\label{thm:maruyama}
Let $X$, $D$, and $\mathscr E \twoheadrightarrow \mathscr F$ be as in \autoref{setup Maruyama}.
Then there is a natural isomorphism of $S$-schemes $\tilde X \stackrel\sim\to \tilde X'$.
\end{theorem}

\begin{proof}
For an $S$-scheme $Z$, write $h_Z$ for the functor $\operatorname{Hom}_S(-,Z)$. We will use the description of
\autoref{cor:universal property} to show that $h_{\tilde X}$ and $h_{\tilde X'}$ are naturally isomorphic,
which implies the result by the Yoneda lemma. Let $q \colon T \to S$ be a morphism of schemes.
Given $(\mathscr L,\phi) \in h_{\tilde X}(T)$, define $(\mathscr L',\phi') \in h_{\tilde X'}(T)$ as the image
\[
\phi' \colon q^* \mathscr E' \twoheadrightarrow \mathscr L' \hookrightarrow \mathscr L,
\]
noting that $\mathscr L'$ is invertible by \autoref{cor:universal property}. This gives a map
$h_{\tilde X} \to h_{\tilde X'}$, and switching the roles of $\mathscr E$ and $\mathscr E'$ using
\autoref{lemma:duality} and the natural isomorphism $\tilde X(-D) \cong \tilde X$ gives the opposite map.
The composition takes $(\mathscr L,\phi)$ to $(\mathscr L(-D),\phi(-D))$, hence under the identification
$\tilde X(-D) \cong \tilde X$ gives the identity map. The other composition follows dually.
\end{proof}
\vskip-\lastskip
This gives the geometric definition of elementary transformations:

\begin{definition}
Let $S$ be a smooth variety, and let $\mathscr E$ and $\mathscr E'$ be vector bundles of the same rank on $S$.
We say that \emph{there exists an elementary transformation between $X = \PP_S(\mathscr E)$ and $X' = \PP_S(\mathscr E')$}
if there exists a smooth divisor $D \subseteq S$, a line bundle $\mathscr L$ on $S$,
and a quotient bundle $\mathscr E|_D \twoheadrightarrow \mathscr F$ on $D$ such that the kernel
of $\mathscr E \twoheadrightarrow \mathscr F$ is $\mathscr E' \otimes \mathscr L$.
In this case, \autoref{thm:maruyama} gives a diagram
\[
\begin{tikzcd}[row sep=.8em,column sep =.2em]
 & \tilde X \ar{ld}[swap]{f}\ar{rd}{f'} & \\
X & & X'
\end{tikzcd}
\]
where $f$ and $f'$ are blowups in the smooth centres $Y \subseteq X$ and $Y' \subseteq X'$ respectively.
\end{definition}

To construct elementary transformations, we will use the following Bertini smoothness theorem for general
sections of a very ample vector bundle.

\begin{theorem}[Kleiman]\label{theorem:Bertini}
Let $X$ be a smooth projective variety of dimension $n$ over a field $k$, let $\mathscr E$ be a globally generated
vector bundle of rank $r$, and let $\mathscr L$ be a very ample line bundle. Then for a general section
$\sigma \in \Gamma(X,\mathscr E \otimes \mathscr L)$, the zero locus $Z(\sigma) \subseteq X$ is smooth of
codimension~$r$ (and nonempty if and only if $r \leq n$).
\end{theorem}

As usual, \emph{general} means that the conclusion holds on a dense Zariski open in the space of sections of $\mathscr E \otimes \mathscr L$. In particular there exists a $k$-point when $k$ is infinite.

\begin{proof}
See \cite[Cor.~3.6 and Rmk.~3.2(iii)]{kleiman}.
\end{proof}
\vskip-\lastskip
This gives the following variant of \cite[Thm.~1.12]{maruyama}:

\begin{proposition}\label{prop:elementary}
Let $S$ be a smooth projective variety of dimension $\leq 2$ over an infinite field $k$, and let $\mathscr E$
and $\mathscr E'$ be vector bundles on $S$ of the same rank $r$. Then there exists an elementary transformation
between $X = \PP_S(\mathscr E)$ and $X' = \PP_S(\mathscr E')$. Moreover, if $Z_1,\ldots,Z_m \subseteq X$
are smooth subvarieties, we may assume that $Y \cap Z_i$ is smooth for all $i$.
\end{proposition}

Maruyama's version deals with the case that $\mathscr E'$ is a trivial bundle, and does not have the final statement.
Note that the final statement is not symmetric in $\mathscr E$ and $\mathscr E'$, and we will apply the result
when $\mathscr E$ is trivial (which is dual to Maruyama's version). Maruyama's result extends to threefolds,
which we will not pursue here.

\begin{proof}[Proof of \autoref{prop:elementary}]
Let $\mathscr L$ be a very ample line bundle on $S$. Up to twisting $\mathscr E'$ by a power of $\mathscr L$
and replacing $\mathscr L$ by a power, we may assume that $(\mathscr E')^\vee \otimes \mathscr L^{-1}$ is
globally generated and $\pi^* \mathscr L \otimes \OO_\pi(1)$ is very ample.
By \autoref{theorem:Bertini} there exists a regular section $\sigma$ of $\pi^*(\mathscr E')^\vee \otimes \OO_\pi(1)$
such that $Y = Z(\sigma)$ and the $Y \cap Z_i$ are smooth.

For each $s \in S$, the intersection $Y \cap X_s$ is given by a section of $\OO_\pi(1)^r$,
hence is a linear subspace. Since $\dim S \leq 2$, we have $\dim Y = (\dim S + r -1) - r \leq 1$.
If $Y$ contains a vertical line $\ell \subseteq X_s$ for some $s \in S$, then $\ell$ is a component of $Y$,
and the normal bundle
\[
\mathscr N_{\ell/X} = \Big(\pi^*(\mathscr E')^\vee \otimes \OO_\pi(1)\Big)\Big|_\ell = \OO_\ell(1)^r
\]
has a trivial quotient $\mathscr N_{X_s/X}|_\ell = \OO_\ell^2$, which is impossible.
Thus if $D \subseteq S$ is the scheme-theoretic image of $Y$ in $S$, then $D$ is reduced and all fibres of $Y \to D$
have length $1$, so $Y \to D$ is an isomorphism. Hence $Y$ corresponds to a section $D \to \PP_D(\mathscr E|_D)$,
i.e.\ a $1$-dimensional quotient $\mathscr E|_D \twoheadrightarrow \mathscr F$.
Since $\sigma$ is a regular section, the Koszul complex
\begin{equation}\label{eq:koszul}
0 \to K_r \to \ldots \to K_1 \to K_0 \to \OO_Y \to 0
\end{equation}
is exact, where $K_i = \pi^*\big(\raisebox{.1em}{$\bigwedge$}^i \mathscr E'\big) \otimes \OO_\pi(-i)$.
The projection formula gives
\[
R\pi_*\big(K_i \otimes \OO_\pi(1)\big) = \raisebox{.1em}{$\textstyle\bigwedge$}^i \mathscr E' \otimes R\pi_*\big(\OO_\pi(-i+1)\big).
\]
Since $\pi \colon X \to S$ is a $\PP^{r-1}$-bundle, we have $R\pi_*\big(\OO_\pi(-i+1)\big) = 0$
for $1 < i \leq r$ and $R^1\pi_* \OO_X = 0$. Therefore, twisting \eqref{eq:koszul} by $\OO_\pi(1)$
and pushing forward to $S$ gives a short exact sequence
\[
0 \to \pi_*\big(K_1 \otimes \OO_\pi(1)\big) \to \pi_*\big(\OO_\pi(1)\big) \to \pi_*\!\left(\OO_\pi(1)\big|_Y\right) \to 0.
\]
Since $K_1 \otimes \OO_\pi(1) = \pi^*\mathscr E'$, this sequence reads
\[
0 \to \mathscr E' \to \mathscr E \to \mathscr F \to 0,
\]
so \autoref{thm:maruyama} gives the desired elementary transformation.
\end{proof}

\bibliographystyle{alphaurledit}
\bibliography{char-p}

\end{document}